\documentclass[12pt]{article}
\usepackage{graphicx}
\usepackage{amssymb, amsmath, amsthm, verbatim}
\usepackage[all]{xy}
\usepackage{epstopdf}
\newcommand{\doublewidealign}[3]{\begin{minipage}[t]{#1}\begin{align*}#2\end{
align*}\end{minipage}\begin{minipage}[t]{2.5in}\begin{align*}#3\end{align*}\end{
minipage}}

\newcommand{\abs}[1]{\left| #1 \right|}
\newcommand{\parens}[1]{\left( #1 \right)}

\newcommand{\varep}{\varepsilon}

\newcommand{\R}{\mathbb{R}}
\newcommand{\RP}{\mathbb{RP}}
\newcommand{\tr}{{\rm tr}}

\newcommand{\C}{\mathbb{C}}

\newcommand{\fund}[2][]{\pi_1^{#1}(#2)}
\newcommand{\slc}[1][\C]{{\rm PSL}_2(#1)}
\newcommand{\slpm}[2][]{{\rm SL}^{\pm}_{#1}(#2)}
\newcommand{\gl}{{\rm GL}}
\newcommand{\liesl}[1][4]{\mathfrak{sl}(#1)}
\newcommand{\lieslact}[1]{\left.\liesl\right._{#1}}
\newcommand{\lieso}[1][3]{\mathfrak{so}(#1,1)}
\newcommand{\liesoact}[1]{\left.\lieso\right._{#1}}
\newcommand{\homol}[3][1]{H^{#1}(#2,#3)}
\newcommand{\vact}[1]{\mathfrak{v}_{#1}}
\newcommand{\pgl}{\rm{PGL}_{4}(\R)}
\newcommand{\PGL}[2][]{\rm{PGL}_{#1}(#2)}
\newcommand{\so}[1]{{\rm SO}(#1,1)}
\newcommand{\rvarpar}[2][\pgl]{\mathcal{R}(#2,#1)_{\rm{p}}}
\newcommand{\rvar}[2][\pgl]{\mathcal{R}(#2,#1)}
\newcommand{\charvar}[2][\pgl]{\mathcal{X}(#2,#1)}
\newcommand{\charvarpar}[2][\pgl]{\mathcal{X}(#2,#1)_{\rm{p}}}
\newcommand{\defspace}[2][]{\mathcal{D}_{#1}(#2)}
\newcommand{\Z}{\mathbb{Z}}
\newcommand{\HH}{\mathbb{H}}

\newcommand{\Ad}[1]{{\rm Ad}(#1)}
\newcommand{\dhilb}[2]{d_{{\rm H}}(#1,#2)}

\newcommand{\bs}{\backslash}

\newcommand{\be}{\begin{enumerate}}
\newcommand{\ee}{\end{enumerate}}
\newcommand{\bd}{\begin{description}}
\newcommand{\ed}{\end{description}}

\newcommand{\Om}{\Omega}
\newcommand{\ol}[1]{\overline{#1}}
\newcommand{\om}{\omega}
\newcommand{\rhogeo}{\rho_{{\rm geo}}}
\newtheorem{theorem}{Theorem}
\newtheorem{proposition}[theorem]{Proposition}
\newtheorem{corollary}[theorem]{Corollary}
\newtheorem{lemma}[theorem]{Lemma}
\newtheorem{remark}[theorem]{Remark}

\numberwithin{equation}{section}
\numberwithin{theorem}{section}
\newtheorem{question}{Question}

\begin{document}

\title{Deformations of Non-Compact, Projective Manifolds}
\author{Samuel A. Ballas}
\maketitle
\begin{abstract}
 In this paper, we demonstrate that the complete hyperbolic structure of various two-bridge knots and links cannot be deformed to an inequivalent strictly convex projective structure. We also prove a complementary result showing that under certain rigidity hypotheses, branched covers of amphicheiral knots admit non-trivial, strictly convex deformations near their complete hyperbolic structure.
\end{abstract}

\section{Introduction}\label{intro}

Mostow rigidity for hyperbolic manifolds is a crucial tool for understanding the deformation theory of lattices in $\rm{Isom}(\HH^n)$. Specifically, it tells us that the fundamental groups of finite volume hyperbolic manifolds of dimension $n\geq3$ admit a unique conjugacy class of discrete, faithful representations into $\rm{Isom}(\HH^n)$.  

Recent work of \cite{CooperLongTillman11, Benoist04I, Benoist05, Benoist06} has revealed several parallels between the geometry of hyperbolic $n$-space and the geometry of strictly convex domains in $\RP^n$. For example, the classification and interaction of isometries of strictly convex domains is analogous to the situation in hyperbolic geometry. Additionally, if the isometry group of the domain is sufficiently large then the strictly convex domain equipped with the Hilbert metric is known to be $\delta$-hyperbolic.  Despite the many parallels between these two types of geometry, there is no analogue of Mostow rigidity for strictly convex domains. This observation prompts the following question: when is it possible to non-trivially deform the complete hyperbolic structure on a finite volume hyperbolic manifold inside the category of strictly convex projective structures? 

Currently, the answer is known only in certain special cases. For example, when the manifold contains a totally geodesic hypersurface there exist non-trivial deformations at the level of representations coming from the bending construction of Johnson and Millson \cite{JohnsonMillson87}. In the closed case, work of Koszul \cite{Koszul68} shows that these new projective structures arising from bending remain properly convex. Further work of Benoist \cite{Benoist05} shows that these structures are actually strictly convex. In the non-compact case, recent work of Marquis \cite{Marquis10} has shown that the projective structures arising from bending remain properly convex in this setting as well.

 In contrast to the previous results, there are examples of closed $3$-manifolds for which no such deformations exist (see \cite{CooperLongThist07}). Additionally, there exist $3$-manifolds that contain no totally geodesic surfaces, that nevertheless admit deformations (see \cite{CooperLongThist06}). Following the terminology in \cite{CooperLongThist07} we refer to deformations that do not arise from the bending construction as \emph{flexing deformations}.
 
Prompted by these results, a natural question is whether or not there exist strictly convex flexing deformations for non-compact finite volume 3-manifolds. Two-bridge knots and links provide a good place to begin exploring because they have particularly simple presentations for their fundamental groups, which makes analyzing their representations a more tractable problem. Additionally, work of \cite{HatcherThurston85} has shown that they contain no closed totally geodesic embedded surfaces.

If $M$ is a manifold, the theory of $(G,X)$ structures tells us that each projective structure on $M$ gives rise to a conjugacy class of representations, $\rho:\fund{M}\to \PGL[n+1]{\R}$ called the \emph{holonomy} of the structure (see \cite{Goldman87} for details). In general, properties of the structure manifest themselves as properties of the holonomy representation. In particular, we will see that there are strong restrictions placed on the holonomy of a strictly convex projective structure.    

While Mostow rigidity guarantees the uniqueness of \emph{complete} structures on finite volume 3-manifolds, work of Thurston \cite[Chap.\ 5]{ThurstonNotes} shows that if we remove the completeness hypothesis then there is an interesting deformation theory for cusped, hyperbolic 3-manifolds. However, these incomplete hyperbolic structures are never strictly (or even properly) convex projective structures. For example, when the completion of the deformed structure corresponds to Dehn filling then the image of the developing map will miss a countable collection of geodesics. 

In order to obtain a complete hyperbolic structure we must insist that the holonomy of the peripheral subgroup be parabolic. In \cite{CooperLongTillman11}, a notion of parabolicity is introduced for automorphisms preserving a properly convex domain. In the strictly convex setting we show that parabolicity of peripheral subgroups is a necessary condition on the holonomy of a strictly convex projective structure (see Lemma \ref{scparabolic}).

Using normal form techniques developed in section \ref{normal forms} we are able to prove that for several two-bridge knot and link complements the holonomy of their complete hyperbolic structure is, locally and up to conjugacy, the only representation with parabolic peripheral holonomy. As a result we are able to prove the following theorem. 

\begin{theorem}\label{rigiditytheorem}
 Let $M$ be the complement in $S^3$ of $4_1$ (the figure-eight knot), $5_2$, $6_1$, or $5_1^2$ (the Whitehead link). Then $M$ does not admit strictly convex deformations of its complete hyperbolic structure\footnote{The notation for these knots and links comes from Rolfsen's table of knots and links \cite{Rolfsen90}}.
\end{theorem}

It should be mentioned that an infinitesimal analogue of this theorem is proven in \cite{PortiHeusener09} for the figure-eight knot and the Whitehead link. In light of Theorem \ref{rigiditytheorem} we ask the following:

\begin{question}\label{twobridgequestion} 
Does any hyperbolic two-bridge knot or link admit a strictly convex deformation of its complete hyperbolic structure? 
\end{question}

In \cite{PortiHeusener09} it is shown that there is a strong relationship between deformations of a cusped hyperbolic 3-manifold and deformations of surgeries on that manifold. In particular, the authors of \cite{PortiHeusener09} use the fact that the figure-eight knot is infinitesimally projectively rigid relative to the boundary to deduce that there are deformations of certain orbifold surgeries of the figure-eight knot. We are able to extend these results to other amphicheiral knot complements that enjoy a certain rigidity property in the following theorem. See sections \ref{localdef} and \ref{rigflex} for the relevant definitions.

\begin{theorem} \label{achiraldef}
 Let $M$ be the complement of a hyperbolic amphicheiral knot, and suppose that $M$ is infinitesimally projectively rigid relative to the boundary and the longitude is a rigid slope. Then for sufficiently large $n$, there is a one-dimensional family of strictly convex deformations of the complete hyperbolic structure on $M(n/0)$.
\end{theorem}

Here $M(n/0)$ is the orbifold obtained by surgering a solid torus with longitudinal singular locus of cone angle $2\pi/n$ along the meridian of $M$. Other than the figure-eight knot, we cannot yet prove that there exist other knots satisfying the hypotheses of Theorem \ref{achiraldef}. However, there is evidence that there should be many situations in which Theorem \ref{achiraldef} applies. First, there is numerical evidence that the two-bridge knot $6_3$ satisfies the hypotheses of Theorem \ref{achiraldef}.  Furthermore, the authors of \cite{CooperLongThist07} examined approximately 4500 closed hyperbolic 3-manifolds and found that only about 1\% of them admit strictly convex deformations of their hyperbolic structure. Combined with Theorem \ref{rigiditytheorem} this suggests that strictly convex deformations of hyperbolic 3-manifolds are somewhat rare.  As a result, two-bridge knots that are infinitesimally rigid relative to the boundary may be quite abundant. Additionally, there are infinitely many 
amphicheiral two-bridge knots, and so there is hope that there are many situations in which Theorem \ref{achiraldef} applies.

The organization of the paper is as follows: section \ref{projective geometry} discusses some basics of projective geometry and projective isometries, while Section \ref{localdef} discusses local and infinitesimal deformations with a focus on bending and its effects on peripheral subgroups. Section \ref{normal forms} discusses some normal forms into which parabolic isometries can be placed. In Section \ref{examples} we use the techniques of the previous section to give a thorough discussion of deformations of the figure-eight knot (Section \ref{fig8section}) and the Whitehead link (Section \ref{whiteheadsection}), which along with computations in \cite{Ballas12}, proves Theorem \ref{rigiditytheorem}. Finally, in Section \ref{rigflex} we discuss some special properties of amphicheiral knots and their representations and prove Theorem \ref{achiraldef}.

\subsection*{Acknowledgements} This project started at the MRC on real projective geometry in June, 2011. I would like to thank Daryl Cooper, who suggested looking at projective deformations of knot complements. I would also like to thank Tarik Aougab for several useful conversations and his help with some of the early computations related to the figure-eight knot. Neil Hoffman also offered useful suggestions on an earlier draft of this paper. Finally, I would like to thank my thesis advisor, Alan Reid for his encouragement and helpful feedback throughout this project.

\section{Projective Geometry and Convex Projective Structures}\label{projective geometry}

Let $V$ be a finite dimensional real vector space. We form the \emph{projectivization of $V$,}
denoted $P(V)$, by dividing  $V\bs\{0\}$ by the action of $\R^\ast$ by scaling. When $V=\R^n$ we
denote $P(\R^n)$ by $\RP^n$. If we take $\gl(V)$ and divide by the action
of $\R^\ast$, acting by multiplication by central linear transformations, we get $\rm{PGL}(V)$. It
is easy to verify that the action of $\gl(V)$ on $V$ descends to an action of
$\rm{PGL}(V)$ on $P(V)$. If $W\subset V$ is a subspace, then we call $P(W)$ a
\emph{projective subspace} of $P(V)$. Note that the codimension of $P(W)$ in
$P(V)$ is the same as the codimension of $W$ in $V$ and the dimension of $P(W)$
is one less than the dimension of $W$. A \emph{projective line} is a
1-dimensional projective subspace.

Let $\Omega$ be a subset of $P(V)$, then $\Omega$ is \emph{convex} if its
intersection with any projective line is connected. An \emph{affine patch} is
the complement in $P(V)$ of a codimension 1 subspace. A convex subset
$\Omega\subset P(V)$ is \emph{properly convex} if its closure, $\overline{\Omega}$, is contained in some
affine patch. A properly convex $\Omega$ of $P(V)$ is \emph{strictly convex} if $\partial \Omega$ does not contain a line segment of positive length in $\partial \Omega$ (here length is measured in the Euclidean metric in some affine patch that contains $\ol{\Omega}$).

We now take a moment to review some examples. An affine patch is a convex subset of $P(V)$, however it is not properly convex as its closure is all of $P(V)$. Let $S=\{(x_1,\ldots x_{n+1})\in \R^{n+1}\vert x_i>0\}$, then the $P(S)$ is a disjoint from the projectivization of the hyperplane $\{(x_1,\ldots, x_{n+1})\in \R^{n+1}\vert \sum_{i=1}^{n+1}x_i=-1\}$ and is thus properly convex. It is straightforward to see that $P(S)$ is a simplex in $\RP^n$ and is thus not strictly convex.

Next, consider the cone, $C=\{(x_1,\ldots,x_{n+1})\in \R^{n+1}\vert x_1^2+\ldots x_{n}^2-x_{n+1}^2<0\}$. $P(C)$ can be identified with $n$-dimensional hyperbolic space (this is the Klein model, see \cite{Ratcliffe06} for more details). $P(C)$ is a disk in $\RP^n$ (when viewed in an appropriate affine patch) and is thus strictly convex.

A projective space $P(V)$ admits a double cover $\pi:S(V)\to P(V)$, where $S(V)$ is
the quotient of $V\bs \{0\}$ by the action of $\R^+$. In the case where $V=\R^n$ we
denote $S(V)$ by $S^n$. The automorphisms of $S(V)$ are $\slpm[]{V}$, which
consists of linear transformations of $V$ with determinant $\pm 1$. Let $[T]\in \PGL[]{V}$ be an equivalence class of linear transformations. By scaling $T$ we can arrange that $T\in \slpm[]{V}$. Additionally, we see that $T\in \slpm[]{V}$ if and only if $-T\in \slpm[]{V}$.  As a result, there is a 2-to-1 map, which by abuse of notation we also call $\pi:\slpm[]{V} \to \PGL[]{V}$ given by $\pi(T)= [T]$.  If we let $\slpm{\Omega}$ and $\PGL{\Omega}$ be subsets of
$\slpm[]{V}$ and $\PGL[]{V}$ preserving $\pi^{-1}(\Omega)$ and $\Omega$, respectively then we see that $\pi$ restricts to a 2-to-1 map from $\slpm[]{\Omega}$ to $\PGL[]{\Omega}$.

When $\Omega$ is properly convex we can construct a section of $\pi$ that is a homomorphism. If $\Omega\subset P(V)$ is properly convex, the preimage of $\Omega$ under $\pi$ will consist of two connected components. Every element of $\slpm[]{\Omega$} either preserves both of these components or interchanges them. Furthermore, $T\in \slpm[]{\Omega}$ preserves both components if and only if $-T$ interchanges them. As a result we see that if $[T]\in \PGL[]{\Omega}$ then there is a unique lift of $[T]$ to $\slpm[]{\Omega}$ that preserves both components of $\pi^{-1}(\Omega)$. Mapping $[T]$ to this lift yields the desired section. Additionally, by using this section we are able to identify $\PGL[]{\Omega}$ with a subgroup of $\slpm[]{\Omega}$. As a result we will regard elements of $\PGL[]{\Omega}$ as linear transformations when convenient. 

We now describe a classification of elements of $\PGL[]{\Omega}$ given in \cite{CooperLongTillman11}. Let $\Omega$ be a properly convex and open subset of $\PGL[]{\Omega}$ and let $T\in \PGL[]{\Omega}$. If $T$ fixes a point in $\Omega$, then $A$ is called \emph{elliptic}. If $T$ acts freely on $\Omega$, and all of its eigenvalues have modulus 1, then $T$ is \emph{parabolic}. Otherwise, $T$ is \emph{hyperbolic}. Since $\Omega$ is properly convex, we can realize $\ol{\Omega}$ as a compact, convex subset of $\R^n$. As a result the Brouwer fixed point theorem tells us that every element of $\PGL[]{\Omega}$ will fix a point in $\ol{\Omega}$. Furthermore, when $\Omega$ is strictly convex, parabolic elements will have a unique fixed point in $\partial \Omega$ and hyperbolic elements will have exactly two fixed points in $\partial \Om$ (see \cite{CooperLongTillman11} Prop 2.8). As a result we see that when $\Omega=\HH^n$ this classification agrees with the standard classification of isometries of hyperbolic space.

Understanding parabolic elements of $\PGL[]{\Omega}$ is an important aspect of the proof of Theorem \ref{rigiditytheorem} and we make crucial use of the following theorem from \cite{CooperLongTillman11}, which demonstrates that parabolic elements of $\PGL[]{\Omega}$ are subject to certain linear algebraic constraints. 

\begin{theorem}[\cite{CooperLongTillman11} Prop 2.9]\label{jnfparabolic}
 Suppose that $\Omega$ is a properly convex domain and that $T\in \PGL[]{\Omega}$
is parabolic, then one of largest Jordan blocks of $T$ has eigenvalue 1.
Additionally, the size of this Jordan block is odd and at least 3. If $\Omega$
is strictly convex then this is the only Jordan block of this size. 
\end{theorem}
  
This theorem is particularly useful in small dimensions. For example, in dimensions 2 or
3 the matrix representing any parabolic element that preserves a properly convex subset is conjugate to either

\[
 \begin{pmatrix}
  1 & 1 & 0\\
  0 & 1 & 1\\
  0 & 0 & 1 
 \end{pmatrix} \rm{ or }
 \begin{pmatrix}
  1 & 0 & 0 & 0\\
  0 & 1 & 1 & 0\\
  0 & 0 & 1 & 1\\
  0 & 0 & 0 & 1
 \end{pmatrix}
\]
respectively. This theorem also tells us that in these small dimensional cases, 1 is the only
possible eigenvalue of a parabolic that preserves a properly convex domain. 

When $\Omega$ is properly convex we can also gain more insight into the structure of
$\PGL[]{\Omega}$ by considering a $\PGL[]{\Omega}$-invariant metric on $\Omega$. Given $x_1,x_2\in
\Omega$ we define the \emph{Hilbert metric} as follows: let $\ell$ be the line
segment in $\Omega$ between $x_1$ and $x_2$. Proper convexity tells us that $\ell$
intersects $\partial \Omega$ in two points $y$ and $z$, where $y$ is the point on $\ell$ closer to $x_1$ and $z$ is the point on $\ell$ closer to $x_2$. We define
$\dhilb{x_1}{x_2}=\log([y:x_1:x_2:z])$, where $[y:x_1:x_2:z]=\frac{\abs{x_2-y}\abs{z-x_1}}{\abs{x_1-y}\abs{z-x2}}$ is the cross ratio. Since the cross ratio is invariant under projective
automorphisms this metric is invariant under
$\PGL[]{\Omega}$. Furthermore, projective lines are geodesics for this metric. The Hilbert metric gives rise to a Finsler structure on $\Omega$, and in
the case that $\Omega$ is an ellipsoid it coincides with twice the standard
hyperbolic metric. 

We now use this metric to understand certain subsets of $\PGL[]{\Omega}$.

\begin{lemma}\label{bdddist}
 Let $\Omega$ be a properly convex domain and let $x$ be a point in the interior
of $\Omega$, then the set $\PGL[]{\Omega}_x^K=\{T\in\PGL[]{\Omega}\mid \dhilb{x}{Tx}\leq
K\}$ is compact.
\end{lemma}

\begin{proof}
 Let $\mathcal{B}=\{x_0,\ldots, x_n\}$ be a projective basis (i.e.\ a set of $n+1$ points, no $n$ of which live in a common projective hyperplane) that is a subset of $\Omega$ and such that $x_0=x$. The group $\PGL[n+1]{R}$ acts simply transitively on the set of projective bases, and the orbit of a fixed projective basis provides a homeomorphism from $\PGL[n+1]{\R}$ to a subset $(\RP^n)^{n+1}$.
 
Next, let $\gamma_i$ be a sequence of elements of $\PGL[]{\Omega}_x^K$. The elements $\gamma_i x_0$ all live in the compact ball of radius $K$ centered at $x$ and so by passing to a subsequence we can assume that $\gamma_i x_0\to x_0^\infty\in \Omega$. Next, we claim that the by passing to a subsequence we can assume that $\gamma_ix_j\to x_j^\infty \in \Omega$ for $1\leq j\leq n$. To see this, observe that 
 $$\dhilb{x_0^\infty}{\gamma_ix_j}\leq\dhilb{x_0^\infty}{\gamma_i x_0}+\dhilb{\gamma_i x_0}{\gamma_i x_j}=\dhilb{x_0^\infty}{\gamma_i x_0}+\dhilb{x_0}{x_j},$$
 and so all of the $\gamma_ix_j$ live in a compact ball centered at $x_0^\infty$.  The proof will be complete if we can show that the set $\{x_0^\infty\ldots,x_{n}^\infty\}$ is a projective basis. Suppose that this set is not a projective basis, then without loss of generality we can assume that the set $\{x_0^\infty,\ldots, x_{n-1}^\infty\}$ is contained in a projective hyperplane. Thus $x_0^\infty$ is contained in the projective plane spanned by $\{x_1^\infty,\ldots, x_{n-1}^\infty\}$. As a result, we can find a point, $y$, in the projective hyperplane spanned by $\{x_1,\ldots, x_{n-1}\}$ such that $\gamma_iy\to x_0^\infty$.  However, since $\mathcal{B}$ is a projective basis we see that 
 $$0<\dhilb{x_0}{y}=\dhilb{\gamma_ix_0}{\gamma_iy},
$$
but
$$
\dhilb{\gamma_ix_0}{\gamma_iy} \to 0,
$$
 which is a contradiction.  
\end{proof}

 As we mentioned previously, isometries of a strictly convex domain interact in ways similar to hyperbolic isometries. As an example of this phenomenon, recall that if $\phi,\psi\in\so{n}$, with $\phi$ hyperbolic, then $\phi$ and $\psi$ cannot generate a discrete subgroup of $\so{n}$ if they
share exactly one fixed point. In particular, parabolic and hyperbolic automorphisms cannot
share fixed points in a discrete group. A similar phenomenon occurs for $\PGL[]{\Om}$, when $\Om$ is strictly convex.

\begin{proposition}\label{discfixedpt}
 Let $\Om$ be a strictly convex domain and $\phi,\psi\in\PGL[]{\Om}$ with $\phi$
hyperbolic. If $\phi$ and $\psi$ have exactly one fixed point in common, then the
subgroup generated by $\phi$ and $\psi$ is not discrete.
\end{proposition}

\begin{proof}
 Notice the similarity between this proof and the proof in  \cite[Thm
5.5.4]{Ratcliffe06}. Suppose for contradiction that the subgroup generated by
$\phi$ and $\psi$ is discrete. Since $\Om$ is strictly convex, $\phi$ has
exactly two fixed points (see \cite[Prop 2.8]{CooperLongTillman11}), $x_1$ and $x_2$. Without loss of generality we can
assume that they correspond to the eigenvalues of smallest and largest modulus,
respectively, and that $\psi$ fixes $x_1$ but not $x_2$. From \cite[Prop
4.6]{CooperLongTillman11} we know that $x_1$ is a $C^1$ point of $\partial \Omega$, and hence there
is a unique supporting hyperplane to $\Om$ at $x_1$ and both $\phi$ and $\psi$
preserve it. This means that there are coordinates with respect to which both $\phi$ and $\psi$ are affine. In particular we can assume that 

\[
 \phi(x)=Ax,\ \psi(x)=Bx+c,
\]
with $c\neq 0$. We now examine the result of conjugating $\psi$ by powers of $\phi$: 
\[
 \phi^{n}\psi\phi^{-n}(x)=A^{n}BA^{-n}x+A^nc
\]
 The fixed point $x_2$ (which had the largest eigenvalue) has now been moved to
the origin and $\phi$ has been projectively scaled so that $x_2$ has eigenvalue 1. Since $\Omega$ is strictly convex, $x_2$ is the unique attracting fixed point, and so after possibly passing to a subsequence we can assume that $\{A^nc\}$ is a sequence of
distinct vectors that converge to the origin. Since
$A^nc=\phi^n\psi\phi^{-n}(0)$ we see that $\{\phi^n\psi\phi^{-n}\}$ is a sequence of distinct automorphisms. The elements of $\{\phi^n\psi\phi^{-n}\}$ all move points on
line between the fixed points of $\phi$ a fixed bounded distance, and so by Lemma
\ref{bdddist} this sequence has a convergent subsequence, which by construction,
is not eventually constant. The existence of such a sequence contradicts
discreteness.   

\end{proof}

The next lemma describes subgroups of $\PGL[]{\Omega}$ that preserve a common geodesic.

\begin{lemma}\label{presgeodesic}
Let $\Om$ be a strictly convex domain and let $\Gamma\leq \PGL[]{\Om}$ be a discrete torsion free subgroup of elements that all preserve a common geodesic in $\Omega$, then $\Gamma$ is infinite cyclic and generated by a hyperbolic element.
\end{lemma}

\begin{proof}
Since the elements of $\Gamma$ all preserve a common geodesic, $\ell\subset \Omega$, there is a homomorphism from $\phi_\ell:\Gamma \to \R$ that assigns to each element its translation length (in the Hilbert metric) along $\ell$. Since $\Gamma$ is torsion free it acts freely on $\Omega$ and thus this map has trivial kernel. The image of $\Gamma$ under $\phi_\ell$ is a discrete subgroup of $\R$ and is thus $\Gamma$ is infinite cyclic. Since $\Gamma$ is torsion free, every element of $\Gamma$ is either hyperbolic or parabolic. However, parabolic elements have unique fixed points and thus cannot preserve a geodesic. Thus we conclude that the generator of $\Gamma$ must be hyperbolic. 

\end{proof}

A peripheral subgroup of a finite volume hyperbolic 3-manifold is isomorphic to a free Abelian group of rank 2. In the following lemma we analyze how such a group can act on a strictly convex domain.

\begin{lemma}\label{scparabolic}
Let $\Omega$ be strictly convex and let $\Gamma\subset \PGL[]{\Omega}$ be a discrete free Abelian subgroup of rank at least 2. If $1 \neq \gamma\in \Gamma$ then $\gamma$ is parabolic. 
\end{lemma}

\begin{proof}
Since $\Gamma$ is discrete and torsion free we see that $\gamma$ is either parabolic or hyperbolic. Suppose for contradiction that $\gamma$ is hyperbolic. Since $\Gamma$ is Abelian we see that $\Gamma$ has a global fixed point on $\partial \Omega$. By Lemma \ref{discfixedpt} we see that every element of $\Gamma$ must have the same fixed point set as $\gamma$, and thus we see that $\Gamma$ preserves the geodesic connecting the two fixed points of $\gamma$. Lemma \ref{presgeodesic} tells us that $\Gamma$ must be cyclic, but that contradicts the fact that $\Gamma$ has rank at least 2.  
\end{proof}

To close this section we briefly describe convex real projective structures on manifolds. For more details about real projective structures and more general $(G,X)$ structures see \cite{Ratcliffe06,ThurstonNotes, Goldman87}. A \emph{real projective structure} on an $n$-manifold $M$ is an atlas of charts $U\to \RP^n$ such that the transition functions are elements of $\PGL[n+1]{\R}$. When equipped with a real projective structure, we call $M$ a \emph{real projective manifold}. We can globalize the data of an atlas by selecting a chart and constructing a local diffeomorphism $D:\tilde M\to \RP^n$ using analytic continuation. This construction also yields a representation $\rho:\fund M \to \PGL[n+1]{\R}$ that is equivariant with respect to $D$. The map and representation are known as a \emph{developing map} and \emph{holonomy representation}, respectively. The only ambiguity in this construction is the choice of initial chart, and different choices of initial charts will result in developing maps which differ 
by 
post composing by an element of $g\in \PGL[n+1]{\R}$. Additionally, the holonomy representations 
will differ by conjugation by $g$. 

When the map $D$ is a diffeomorphism onto a properly (resp.\ strictly) convex set, $\Omega$, we say that the real projective structure is \emph{properly (resp.\ strictly) convex}. In this case the holonomy representation is both discrete and faithful and $M\cong \Omega/\rho(\Gamma)$, where $\rho$ is the holonomy representation. A key example to keep in mind are complete hyperbolic structures.
\section{Local and Infinitesimal Deformations}\label{localdef}

Unless explicitly mentioned, $\Gamma$ will henceforth denote the fundamental group of a finite volume hyperbolic 3-manifold. By Mostow rigidity, there is a unique conjugacy class of representations of $\Gamma$ that is faithful and has discrete image in $\so{3}$. We call this class the \emph{geometric representation of $\Gamma$} and denote it $[\rhogeo]$. From the previous section we know that real projective structures on a manifold give rise to conjugacy classes of representations of its fundamental group into $\PGL[4]{\R}$. As we shall see, if we want to study real projective structures near the complete hyperbolic structure it suffices to study conjugacy classes of representations near $[\rhogeo]$. 

We now set some notation. Let $\rvar{\Gamma}$ be the $\pgl$ \emph{representation variety} of $\Gamma$. The group $\pgl$ acts on $\rvar{\Gamma}$ by conjugation, and the quotient by this action is the \emph{character variety}, which we denote $\charvar{\Gamma}$.  The character variety is not globally a variety because of pathologies of the action by conjugation, however at $\rhogeo$ the action is nice enough to guarantee that $\charvar{\Gamma}$ has the local structure of a variety.  

Next, we define a refinement of $\charvar{\Gamma}$ that better controls the representations on the boundary. Recall, that by Theorem \ref{jnfparabolic} the only conjugacy class of parabolic element that is capable of preserving a properly convex domain is
\begin{equation}\label{so31parabolic} \begin{pmatrix}
  1 & 0 & 0 & 0\\
  0 & 1 & 1 & 0\\
  0 & 0 & 1 & 1\\
  0 & 0 & 0 & 1
 \end{pmatrix}.
 \end{equation}
Because this is the conjugacy class of a parabolic element of ${\rm SO}(3,1)$ we will call parabolic elements conjugate to \eqref{so31parabolic} \emph{$\so{3}$-parabolic}. Let $\rvarpar{\Gamma}$ be the elements of $\rvar{\Gamma}$ such that peripheral elements of $\Gamma$ are mapped to $\so{3}$-parabolic elements, and let the \emph{relative character variety}, denoted $\charvarpar{\Gamma}$, be the corresponding quotient by the action $\pgl$. 

If $\rho$ is a representation, then a \emph{deformation of $\rho$} is a smooth
map, $\sigma(t): (-\varep,\varep)\to \rvar{\Gamma}$ such that
$\sigma(0)=\rho$. Often times we will denote $\sigma(t)$ by $\sigma_t$. If a
class $[\sigma]$ is an isolated point of the $\charvar{\Gamma}$ then we say that
$\Gamma$ is \emph{locally projectively rigid at $\sigma$}. Similarly, we say that $\Gamma$ is \emph{locally projectively rigid  relative to the boundary at $\sigma$} when $[\sigma]$ is isolated in $\charvarpar{\Gamma}$. 

We now relate these varieties to projective structures. For more details about geometric structures on manifolds see \cite[\S 2 and \S3]{Goldman87}. Let $M$ be a 3-manifold with $\Gamma=\fund{M}$.  We define the space of projective structures on $M$, denoted $\defspace{M}$ as the set of equivalence classes 
$$\{(f,N)\vert \text{$N$ is a real projective manifold and $f:M\to N$ is a diffeomorphism}\}/\sim$$
where $(f,N)\sim (f',N')$ if there exists a projective transformation, $g$, defined on the complement of a collar neighborhood of $\partial N$ onto the complement of a collar neighborhood of $\partial N'$ such that $g\circ f$ is isotopic to $f'$. The set $\defspace{M}$ consists of the different projective structures on $M$ and a continuous path in $\defspace{M}$ is a \emph{projective deformation of $M$}.

 As we have seen, a projective manifold $N$ diffeomorphic to $M$ gives rise to an element  $[\rho_N]\in \charvar{\Gamma}$ via its holonomy representation. As a result we have a map $hol:\defspace{M}\to \charvar{\Gamma}$ given by $hol([(f,N)])=[\rho_N\circ f_\ast]$. The following specific instance of a much more general theorem allows us to understand projective deformations in terms of representations. 

\begin{theorem}[{\cite[\S 3]{Goldman87}}]\label{thurston-lok}
The map $hol$ is a local homeomorphism. 
\end{theorem}

Let $\defspace[sc]{M}$ be the set of classes $[(f,N)]\in \defspace{M}$ such that $N$ is a strictly convex projective manifold. This set is comprised of the different strictly convex projective structures on $M$ and a continuous path in $\defspace[sc]{M}$ is a \emph{strictly convex deformation of $M$}. When $M$ is a finite volume hyperbolic 3-manifold then its peripheral subgroups are all free Abelian groups of rank 2. Thus, as a consequence of Lemma \ref{scparabolic} we have the following proposition about the restriction of $hol$ to $\defspace[sc]{M}$.

\begin{proposition}\label{peripheralholonomy}
If $M$ is a finite volume hyperbolic 3-manifold then $hol(\defspace[sc]{M})\subset \charvarpar{\Gamma}$.  
\end{proposition}

As a consequence of Theorem \ref{thurston-lok} and Proposition \ref{peripheralholonomy} we have the following corollary.

\begin{corollary}\label{scdef}
Let $M$ be a finite volume hyperbolic 3-manifold. If $\Gamma$ is locally projectively rigid relative to the boundary at $[\rhogeo]$ then there are no strictly convex deformations near the complete hyperbolic structure on $M$.
\end{corollary}

\subsection{Twisted Cohomology and Infinitesimal Deformations}

We now review how the cohomology
of $\Gamma$ with a certain system of local coefficients helps to
infinitesimally parameterize conjugacy classes of deformations of
representations. For more details on cohomology see \cite[Chap III]{Brown82}. 

Let $\sigma_0:\Gamma\to \pgl$ be a representation, then we can define the cochain complex, $C^n(\Gamma,\lieslact{\sigma_0})$ to be the set of all functions from $\Gamma^n$ to $\liesl$. The chain complex is equipped with a differential, $d^n:C^n(\Gamma,\lieslact{\sigma_0})\to C^{n+1}(\Gamma,\lieslact{\rho_0})$, where $d^n\phi(\gamma_1,\ldots, \gamma_{n+1})$ is given by  
$$
\gamma_1\cdot \phi(\gamma_2,\ldots,\gamma_{n+1})+\sum_{i=1}^n(-1)^i\phi(\gamma_1,\ldots, \gamma_{i-1},\gamma_i\gamma_{i+1},\ldots,\gamma_{n+1})+(-1)^{n+1}\phi(\gamma_1,\ldots,\gamma_{n}),
$$
where $\gamma\in \Gamma$ acts via the adjoint action of $\sigma_0$, given by $\gamma\cdot
M=\sigma_0(\gamma)M\sigma_0(\gamma)^{-1}$. Letting $Z^n(\Gamma,\lieslact{\sigma_0})$ and $B^n(\Gamma,\lieslact{\sigma_0})$ be the kernel of $d^n$ and image of $d^{n-1}$, respectively, we can form the cohomology groups $\homol[\ast]{\Gamma}{\lieslact{\sigma_0}}$. When no confusion will arise we often omit the superscript and just write $d$ for the boundary map. 

To see how this construction is related to deformations, let $\sigma_t$ be a deformation of $\sigma_0$, then since $\pgl$ is a Lie group
we can use a series expansion and write $\sigma_t(\gamma)$ as
\begin{equation}\label{seriesexpansion}
 \sigma_t(\gamma)=(I+z(\gamma)t+O(t^2))\sigma_0(\gamma ),
\end{equation}
where $\gamma\in\Gamma$ and $z$ is a map from $\Gamma$ into $\liesl$, which we
call an \emph{infinitesimal deformation of $\sigma_0$}. The above construction would have worked for any smooth function from $\R$ to $\pgl$, but the fact that $\sigma_t$ is a
homomorphism for each $t$ puts strong restrictions on $z$. Let $\gamma$ and
$\gamma'$ be elements of $\Gamma$, then 
\begin{equation}\label{cocyclecond}
 \sigma_t(\gamma\gamma')=(I+z(\gamma\gamma')t+O(t^2))\sigma_0(\gamma\gamma')
\end{equation}
and 
\begin{equation*}
 \sigma_t(\gamma)\sigma_t(\gamma')=(I+z(\gamma)t+O(t^2))\sigma_0(\gamma)(I+z(\gamma')+O(t^2))\sigma_0(\gamma')
\end{equation*}
\begin{equation*}
=(I+(z(\gamma)+\gamma\cdot
z(\gamma'))t+O(t^2))\sigma_0(\gamma\gamma').
\end{equation*}

By focusing on the linear terms of the two power series in \eqref{cocyclecond}, we find that
$z(\gamma\gamma')=z(\gamma)+\gamma\cdot z(\gamma')$, and so $z\in
Z^1(\Gamma,\left.\liesl\right._{\sigma_0})$, the set of 1-cocycles twisted by
the action of $\sigma_0$. 

Since we think of deformations coming from conjugation as uninteresting, we now analyze which elements in
$Z^1(\Gamma,\left.\liesl\right._{\sigma_0})$ arise from these types of
deformations. Let $c_t$ be a smooth curve in $\pgl$ such that $c_0=I$, and
consider the deformation $\sigma_t(\gamma)=c^{-1}_t\sigma_{0}(\gamma)c_t$. Again, when
we write $\sigma_t(\gamma)$ as a power series we find that 
\begin{equation}\label{cobound}
c_t^{-1}\sigma_0(\gamma)c_t=(I-z_ct+O(t^2))\sigma_0(\gamma)(I+z_ct+O(t^2))=(I+(z_c-\gamma\cdot z_c)t+O(t^2))\sigma_0(\gamma),
\end{equation}
and again by looking at linear terms we learn $z(\gamma)=z_c-\gamma\cdot z_c$ and so $z$ is a 1-coboundary. 

We therefore conclude that studying infinitesimal deformations near $\sigma_0$ up
to conjugacy boils down to studying
$\homol{\Gamma}{\left.\liesl\right._{\sigma_0}}$. In the case where
$\homol{\Gamma}{\lieslact{\sigma_0}}=0$ we say that $\Gamma$ is
\emph{infinitesimally projectively rigid at $[\sigma_0]$}, and when the map $\iota^\ast:\homol{\Gamma}{\lieslact{\sigma_0}}\to\homol{\fund{M}}{\lieslact{\sigma_0}}$ induced by the inclusion $\iota:\partial M \to M$ is injective we say that $\Gamma$ is \emph{infinitesimally projectively rigid relative to the boundary at $[\sigma_0]$}. The following theorem of Weil shows the strong relationship between infinitesimal and local
rigidity. 

\begin{theorem}[\cite{Weil64}]\label{infrig}
 If $\Gamma$ is infinitesimally projectively rigid at $[\sigma]$ then $\Gamma$ is locally projectively rigid at $[\sigma]$. 
\end{theorem}
More generally the dimension of $\homol{\Gamma}{\lieslact{\rhogeo}}$ is an upper bound for the dimension of $\charvar{\Gamma}$ at $[\rhogeo]$ (see \cite[Sec 2]{JohnsonMillson87}). However, it is important to remember that in general the character variety need not be a smooth manifold. When this occurs this bound is not sharp and so the converse to Theorem \ref{infrig} is in general false. 

\subsection{Decomposing $\homol{\Gamma}{\lieslact{\rho_0}}$}
  
In order to simplify the study of $\homol{\Gamma}{\lieslact{\rho_0}}$ we will
decompose it into two factors using a decomposition of $\liesl$. Consider the symmetric matrix 
\[
 J=\begin{pmatrix}
    1& 0 & 0 & 0\\
    0 & 1 & 0 & 0 \\
    0 & 0 & 1 & 0\\
    0 & 0 & 0 & -1
   \end{pmatrix}
\]
Following \cite{JohnsonMillson87,PortiHeusener09} we see that this matrix gives rise to
the following decomposition of $\liesl$ as $\rm{PSO}(3,1)$-modules: 
\begin{equation}\label{slsplitting}
 \liesl=\lieso\oplus \mathfrak{v},
\end{equation}
where $\lieso=\{a \in \liesl\mid a^tJ=-Ja\}$ and $v=\{a\in \liesl \mid
a^tJ=Ja\}$. These two spaces are the $\pm 1$-eigenspaces of the
involution $a\mapsto -Ja^tJ$. This splitting of $\liesl$ gives rise to a
splitting of the cohomology groups, namely
\begin{equation}\label{splitting}
 H^\ast(\Gamma,\lieslact{\rho_0})=H^\ast(\Gamma,\liesoact{\rho_0})\oplus
H^\ast(\Gamma,\vact{\rho_0}).
\end{equation}
If $\rhogeo$ is the geometric representation then the first factor of \eqref{splitting} is well understood. By work of Garland \cite{Garland67}, $\homol{M}{\liesoact{\rhogeo}}$ injects into $\homol{\partial M}{\liesoact{\rhogeo}}$. Furthermore, we have
 $\dim{\homol{M}{\liesoact{\rhogeo}}}=2k$, where $k$ is the number of cusps of $M$ (for more details see \cite{Porti97} or \cite[Section 8.8]{Kapovich01}).

Now that we understand the structure of $\homol{\Gamma}{\liesoact{\rho_0}}$ we
turn our attention to the other factor of \eqref{splitting}. The inclusion $\iota:\partial M \to M$ induces a map $\iota^\ast:\homol{M}{\lieslact{\rhogeo}}\to \homol{\partial M}{\lieslact{\rhogeo}}$ on cohomology. We will refer to the kernel of this map as the \emph{$\lieslact{\rhogeo}$-cuspidal cohomology}. In \cite{PortiHeusener09}, Porti and
Heusener analyze the image of this map. The portion that we will use can be
summarized by the following theorem, which can be thought of as a twisted cohomology analogue of the classical half lives/half dies theorem. 
\begin{theorem}[{\cite[Lem 5.3 \& Lem 5.8]{PortiHeusener09}}]\label{homologyimage}
 Let $\rhogeo$ be the geometric representation of a finite volume hyperbolic
3-manifold with k cusps, then
$\rm{dim}(\iota^\ast(\homol{\Gamma}{\vact{\rhogeo}}))$=k. Furthermore, if
$\partial M=\sqcup_{i=1}^k\partial M_i$, then there exists
$\gamma=\sqcup_{i=1}^k \gamma_i$ with $\gamma_i\subset \partial M_i$, such that
$\iota^\ast(\homol{\Gamma}{\vact{\rhogeo}})$ injects into
$\bigoplus_{i=1}^k\homol{\gamma_i}{\vact{\rhogeo}}$.
\end{theorem}

\subsection{Bending} 
Now that we have some upper bounds on the dimensions of our deformations spaces
we want to begin to understand what types of deformations exist near the
geometric representation. The most well know construction of such
deformations is bending along a totally geodesic hypersurface. The goal of this section is to explain the bending construction and prove the following theorem about the effects of bending on the peripheral subgroups.

\begin{theorem}\label{bendingperipheral}
Let $M$ be a finite volume hyperbolic manifold, let $S$ be a totally geodesic hypersurface, and let $[\rho_t]\in \charvar[\rm{PGL}_{n+1}(\R)]{\Gamma}$, obtained by bending along $S$. Then $[\rho_t]$ is contained in $\charvarpar[\rm{PGL}_{n+1}(\R)]{\Gamma}$ if and only if $S$ is closed or each curve dual to the intersection of $\partial M$ and $S$ has zero signed intersection with $S$. 
\end{theorem}

Following \cite{JohnsonMillson87}, let $M$ be a finite volume hyperbolic manifold of dimension $n\geq3$ and $\rhogeo$ be its geometric representation. Next, let  $S$ be a properly embedded
totally geodesic hypersurface. Such a hypersurface gives rise to a non-trivial curve in $\charvar[{\PGL[n+1]{\R}}]{\Gamma}$ passing through $\rho_0$ as follows. We begin with a lemma showing that the hypersurface $S$ gives rise to an element of $\liesl[n+1]$ that is invariant under $\Delta=\pi_1(S)$ but not all of $\Gamma=\pi_1(M)$. 

\begin{lemma} \label{semisimple}
Let $M$ and $S$ as above, then there exists a unique 1-dimensional subspace of
$\liesl[n+1]$ that is invariant under the adjoint action of $\rhogeo(\Delta)$. Furthermore this subspace is generated by a conjugate in $PGL_{n+1}(\R)$ of  
\[
\begin{pmatrix}
-n & 0\\
0 & I
\end{pmatrix},
\]
where $I$ is the $n\times n$ identity matrix.
\end{lemma} 

\begin{proof}
$\rhogeo(\Gamma)$ is a subgroup of $PO(n,1)$ (the projective orthogonal group of the form
$x_1^2+x_2^2+\ldots x_n^2-x_{n+1}^2$). Since $S$ is totally geodesic we can
assume, after conjugation, that $\rhogeo(\Delta)$ preserves both the hyperplane where
$x_1=0$ and its orthogonal complement which is generated by $(1,0,\ldots,0)$.
Hence if $A\in \rhogeo(\Delta)$ then 
$$A=\left(
\begin{array}{cc}
1& 0^{T}\\
0 & \tilde A
\end{array}
\right),$$
where $\tilde A\in PO(n-1,1)$ (the projective orthogonal group of the form
$x_2^2+x_3^2+\ldots+x_n^2-x_{n+1}^2$)  and $0\in \R^n$. If $x\in\liesl[n+1]$ is invariant
under the adjoint action of $\rhogeo(\Delta)$ then we know that $B(t)=\exp{(tx)}$ commutes with every  $A\in\rhogeo(\Delta)$.
If we write $B(t)=\left(\begin{array}{cc}b_{11}& b_{12}\\ b_{21}&
b_{22}\end{array}\right),$ where $b_{11}\in\R$, $b_{12}^T,b_{21}\in\R^n$, and
$b_{22}\in \rm{SL}_n(\R)$, then 
$$\left(
\begin{array}{cc}
b_{11} & b_{12}\\
\tilde Ab_{21} & \tilde Ab_{22}
\end{array}
\right)
=\left(\begin{array}{cc}\
1 & 0\\
0 & \tilde A
\end{array}\right)
\left(
\begin{array}{cc}
b_{11} & b_{12}\\
b_{21} & b_{22}
\end{array}
\right)=
\left(
\begin{array}{cc}
b_{11} & b_{12}\\
b_{21} & b_{22}
\end{array}
\right)
\left(\begin{array}{cc}\
1 & 0\\
0 & \tilde A
\end{array}
\right)=
\left(
\begin{array}{cc}
b_{11} & b_{12}\tilde A\\
b_{21} & b_{22}\tilde A
\end{array}
\right)$$
From this computation we learn that the vectors $b_{12}$ and $b_{21}$ are invariants under
$\rhogeo(\Delta)$ and that $b_{22}$ is in the centralizer of $\rhogeo(\Delta)$ in $SL_{n}(\R)$.
However, $\rhogeo(\Delta)$ is an irreducible subgroup of $PO(n-1,1)$, and so
the only matrices that commute with every element of $\rhogeo(\Delta)$ are scalar
matrices and the only invariant vector of $\rhogeo(\Delta)$ is 0, and so 
$$B=\left(
\begin{array}{cc}
e^{-n \lambda t} & 0\\
0 & e^{ \lambda t} I
\end{array}
\right)$$
where $I$ is the identity matrix. Differentiating $B(t)$ at $t=0$ we find that 
$$x=\left(
\begin{array}{cc}
-n \lambda & 0\\
0&  \lambda I
\end{array}
\right)$$
and the result follows.  
  
\end{proof}

A generator of the $\rhogeo(\Delta)$-invariant subspace of $\liesl[n+1]$ constructed in Lemma \ref{semisimple} will be called a \emph{bending cocycle}. We choose such a generator once and for all and denote it $x_S$. We can now define a family of deformations of $\rhogeo$. The construction breaks into two cases depending on whether or not $S$ is separating.

If $S$ is separating then $\Gamma$ splits as the following amalgamated free product:
 $$\Gamma\cong \Gamma_1\ast_{\Delta}\Gamma_2,$$
where $\Gamma_i$ are the fundamental groups of the components of the complement
of $S$ in $M$. Since $\rhogeo(\Gamma)$ is irreducible we know that $x_S$ is not invariant under all of $\rhogeo(\Gamma)$ and so we can assume without loss
of generality that it is not invariant under $\Gamma_2$. So let
$\rho_t|_{\Gamma_1}=\rhogeo$ and $\rho_t|_{\Gamma_2}=\Ad{\exp{(tx_S)}}\cdot\rhogeo$.
Since these two maps agree on $\Delta$ they give a well defined family of
homomorphisms of $\Gamma$, such that $\rho_0=\rhogeo$. 

If $S$ is nonseparating, then $\Gamma$ is realized as the following HNN extension:
$$\Gamma\cong \Gamma'\ast_\Delta,$$
where $\Gamma'$ is the fundamental group of $M\bs S$. If we let $\alpha$ be a
curve dual to $S$ then we can define a family of homomorphisms through $\rhogeo$ by
$\rho_t|_{\Gamma'}=\rhogeo$ and $\rho_t(\alpha)=\exp{(tx_S)}\rhogeo(\alpha)$. Since
$x_S$ is invariant under $\rhogeo(\Delta)$ the values of $\rho_t(\iota_1(\Delta))$ do not
depend on $t$, where $\iota_1$ is the inclusion of the positive boundary
component of a regular neighborhood of $S$ into $M\bs S$, and so we have well
defined homomorphisms of the HNN extension. 

In both cases $\rho_t$ gives rise to a non-trivial curve of representations and by examining the class of $[\rho_t] \in \homol{\Gamma}{\liesl[n+1]_{\rhogeo}}$,
Johnson and Millson \cite{JohnsonMillson87} showed that $[\rho_t]$ actually
defines a non-trivial path in $\charvar[\rm{PGL}_{n+1}(\R)]{\Gamma}$.

We can now prove Theorem \ref{bendingperipheral}. 

\begin{proof}[Proof of Theorem \ref{bendingperipheral}]
 
In the case that $S$ is closed and separating a peripheral element, $\gamma\in\Gamma$, is contained in either $\Gamma_1$ or $\Gamma_2$ since it is disjoint from $S$.  In this case
$\rho_t(\gamma)$ is either $\rho(\gamma)$ or some conjugate of $\rho(\gamma)$.
In either case we have not changed the conjugacy class of any peripheral
elements and so $[\rho_t]$ is a curve in
$\charvarpar[\rm{PGL}_{n+1}(\R)]{\Gamma}$. Similarly if $S$ is closed and
non-separating then if $\gamma\in \Gamma$ is peripheral we see that $\gamma\in \pi_1(M\bs S)$, and so its
conjugacy class does not depend on $t$. 

In the case that $S$ is non-compact we need to analyze its intersection with the boundary of $M$ more carefully. First, let $M'$ be a finite sheeted cover of $M$ and $S'$ a lift of $S$ in $M'$. If bending along $S$ preserves the peripheral structure of $M$ then bending along $S'$ will preserve the peripheral structure of $M'$. Therefore, without loss of generality we may pass to a finite sheeted cover of $M$ such that $\partial M'=\sqcup_{i=1}^kT_i$, where each of the $T_i$ is a torus (this is possible by work of \cite{McReynoldsReidStover12}).  By looking in the universal cover, it is easy to see that since $S$ is properly embedded and totally geodesic that for any boundary component $T_i$, $T_i\cap S=\sqcup_{j=1}^l t_j$, where the $t_j's$ are parallel $n-2$ dimensional tori. The result of the bending deformation on the boundary will be to simultaneously bend along all of these parallel tori, however some of the parallel tori may bend in opposite directions and can sometimes cancel with 
one 
another. In fact, if we look at a curve $\alpha \in T_i$ 
dual to one (hence all) of the parallel tori, then intersection points of $\alpha$ with $S$ are in bijective correspondence with the $t_j's$, and the direction of the bending corresponds to the signed intersection number of $\alpha$ with $S$. 

When $S$ is separating the signed intersection number of $\alpha$ with $S$ is always zero, and so bending along $S$ has no effect on the boundary. When $S$ is non-separating then the signed intersection can be either zero of non-zero. In the case where the intersection number is non-zero the boundary becomes non-parabolic after bending. This can be seen easily by looking at the eigenvalues of peripheral elements after they have been bent.

\end{proof}  

\begin{remark}
 The Whitehead link is an example where bending along a non-separating totally geodesic surface has no effect on the peripheral subgroup of one of the components (See section \ref{whiteheadsection})
\end{remark}

\begin{corollary}\label{nototgeosurf}
 Let $M$ be a finite volume hyperbolic 3-manifold. If $M$ is locally projectively rigid relative the boundary near the geometric representation then $M$ contains no closed embedded totally geodesic surfaces or finite volume embedded separating totally geodesic surfaces. 
\end{corollary}

The Menasco-Reid conjecture \cite{MenascoReid92} asserts that hyperbolic knot complements do not contain closed, totally geodesic surfaces. As a consequence of Corollary \ref{nototgeosurf} we see that any hyperbolic knot complement that is locally projectively rigid relative to the boundary near the geometric representation will satisfy the conclusion of the Menasco-Reid conjecture.

\section{Some Normal Forms} \label{normal forms}

The goal of this section will be to examine various normal forms into which two non-commuting $\so{3}$-parabolic elements can be placed. In \cite{Riley72}, Riley shows how two non-commuting parabolic elements $a$ and $b$ inside of $\slc$ can be simultaneously conjugated into the following form:
\begin{equation}\label{slcnormalform}
 a=\begin{pmatrix}
    1 & 1\\
    0 & 1
   \end{pmatrix}, \hspace{.25 in }
 b=\begin{pmatrix}
    1 & 0\\
    \om & 1
   \end{pmatrix}.
\end{equation}
In this same spirit we would like to take two $\so{3}$-parabolic elements $A$ and $B$ that are
sufficiently close to the $\so{3}$-parabolic elements $A_0$ and $B_0$ in \emph{the same} copy of  $\so{3}$ and show that $A$ and $B$ can be simultaneously conjugated into the following normal form, similar to \eqref{slcnormalform}.  
\begin{equation}\label{normform1}
A=\parens{\begin{array}{cccc}
1 & 0 & 2 & 1+a_{14}\\
0 & 1 & 2 & 1 \\
0 & 0 & 1 & 1 \\
0 & 0 & 0 & 1
\end{array}
} \hspace{.25 in}
B=\parens{\begin{array}{cccc}
1 & 0 & 0 & 0\\
b_{21} & 1 & 0 & 0\\
b_{31}+b_{21} b_{32} & 2b_{32} & 1 & 0\\
b_{21}+b_{41} & 2 & 0 & 1
\end{array}
}
\end{equation} 
Our goal will be to build a homomorphism from $\slc$ to $\PGL[4]{\R}$ that maps a pair of parabolic elements of $\slc$ of the form \eqref{slcnormalform} to a pair of elements of $\pgl$ of the form \eqref{normform1}. 

Let $b_{32}$ and $b_{21}$ be real numbers, let $d=b_{21}-4b_{32}^2$, and let $(x,y,z,t)$ be coordinates for $\R^4$. Let $Herm_2$ be the real vector space of $2\times 2$ Hermitian matrices. Then there is a linear isomorphism from $\R^4$ to $Herm_2$ given by 
\begin{equation}\label{r4herm}(x,y,z,t)\mapsto
\end{equation}
$$\parens{\begin{array}{cc}
 \mbox{\footnotesize $x$} & \mbox{\footnotesize $x-y+2b_{32}z-2b_{32}^2t+i(\sqrt{d}z-b_{32}\sqrt{d}t)$}\\
\mbox{\footnotesize $x-y+2b_{32}z-2b_{32}^2t-i(\sqrt{d}z-b_{32}\sqrt{d}t)$} & \mbox{\footnotesize $dt$}
\end{array}}
$$

Furthermore there is an action of $\text{SL}_2(\C)$ on $Herm_2$ (and thus on $\R^4$) by linear automorphism given by $M\cdot N=MNM^\ast$, where $M\in \text{SL}_2(\C)$, $N\in Herm_2$, and $\ast$ denotes the conjugate transpose operator. It is easy to see that the kernel of this action is $\{\pm I\}$ and as a result we get homomorphism $\phi':\slc\to \PGL[4]{\R}$.

The function $-Det$ gives us a quadratic form on $Herm_2$ which when regarded as a quadratic form on $\R^4$ via \eqref{r4herm} is induced by the matrix

\begin{equation}\label{quadform}
X=\begin{pmatrix}
1 & -1 & 2b_{32} & -b_{21}/2\\
-1 & 1 & -2b_{32} & 2b_{32}^2\\
2b_{32} & -2b_{32} & b_{21} & -b_{21}b_{32}\\
-b_{21}/2 & 2b_{32}^2 & -b_{21}b_{32} & b_{21}b_{32}^2
	\end{pmatrix}
\end{equation}
By construction, the image of $\phi'$ preserves this form and when $d>0$ we see that this form has signature $(3,1)$. A simple computation shows that 

$$\phi'\parens{\begin{pmatrix}
1 & i\sqrt{d}\\
0 & 1
\end{pmatrix}}=
\begin{pmatrix}
1 & 0 & 2 & 1-2b_{32}\\
0 & 1 & 2 & 1\\
0 & 0 & 1 & 1\\
0 & 0 & 0 & 1
\end{pmatrix}$$

By pre-composing with conjugation by an appropriate diagonal element in $\slc$ we get a new homomorphism $\phi:\slc\to \PGL[4]{\R}$ such that 

$$\phi\parens{\begin{pmatrix}
1 & 1\\
0 & 1
\end{pmatrix}}=
\begin{pmatrix}
1 & 0 & 2 & 1-2b_{32}\\
0 & 1 & 2 & 1\\
0 & 0 & 1 & 1\\
0 & 0 & 0 & 1
\end{pmatrix}$$

When $b_{21}=\abs{\omega}^2$ and $b_{32}=Re(\omega)/2$ we see that $\phi$ will take pairs of the form \eqref{slcnormalform} to pairs of the form \eqref{normform1} (in this case $a_{13}=-Re(\omega)$, $b_{31}=0$ and $b_{41}=-2$). 

With these
assumptions on $b_{21}$ and $b_{32}$ we see that
\[
 d=b_{21}-4b_{32}^2=\abs{\om}^2-\rm{Re}(\om)^2=\rm{Im}(\om)^2.
\]
Thus as long as $\rm{Im}(\om)\neq 0$, $\phi(a)$ and $\phi(b)$ will preserve a common form of signature $(3,1)$.

Let $F_2=\langle \alpha,\beta \rangle$ be the free
group on two letters and let $\rvarpar{F_2}$ be the set of homomorphisms, $f$,
from $F_2$ to $\pgl$  such that $f(\alpha)$ and $f(\beta)$ are $\so{3}$-parabolic elements. The topology of $\rvarpar{F_2}$ is induced by convergence of the
generators. There is a natural action of $\pgl$ on $\rvarpar{F_2}$ by
conjugation, and we denote the quotient of this action by $\charvarpar{F_2}$. 
We have the following lemma about the local structure of $\charvarpar{F_2}$. 
   
\begin{lemma}\label{normformlemma1}
Let $f_0\in \rvarpar{F_2}$ satisfy the following conditions:
\begin{enumerate}
 \item $\langle f_0(\alpha),f_0(\beta)\rangle$ is irreducible and conjugate into $\so{3}$.
 \item $\langle f_0(\alpha),f_0(\beta)\rangle$ is not conjugate into $\so{2}$.
\end{enumerate}
Then for $f\in\rvarpar{F_2}$ sufficiently close to $f_0$  there exists a unique (up to $\pm I$)
element $G_f\in\PGL[4]{\R}$ such that

$$
G_f^{-1}f(\alpha)G_f=\parens{\begin{array}{cccc}
1 & 0 & 2 & 1+a_{14}\\
0 & 1 & 2 & 1 \\
0 & 0 & 1 & 1 \\
0 & 0 & 0 & 1
\end{array}
}
\hspace{.25 in}
G_f^{-1}f(\beta)G_f=\parens{\begin{array}{cccc}
1 & 0 & 0 & 0\\
b_{21} & 1 & 0 & 0\\
b_{31}+b_{21} b_{32} & 2b_{32} & 1 & 0\\
b_{21}+b_{41} & 2 & 0 & 1
\end{array}
}
$$
Additionally, the map from $\rvarpar{F_2}$ to itself given by  $f\mapsto G_f^{-1}fG_f$ is continuous. 
\end{lemma}

\begin{proof}
 The previous argument combined with properties 1 and 2 ensure that $f_0(\alpha)$ and $f_0(\beta)$ can be put into the form \eqref{normform1}. Let $A=f(\alpha)$ and $B=f(\beta)$. Let $E_A$ and $E_B$ be the 1-eigenspaces of
$A$ and $B$, respectively. Since both $A$ and $B$ are $\so{3}$-parabolics both
of these spaces are 2-dimensional. Irreducibility is an open condition, and so we can assume that $f$ is also irreducible and so $E_A$ and $E_B$ have trivial intersection. Therefore $\R^4=E_A\oplus E_B$. If we
select a basis with respect to this decomposition then our matrices will be of
the following block form.
\[
 \begin{pmatrix}
  I & A_U\\
  0 & A_L
 \end{pmatrix}\hspace{.25 in}
 \begin{pmatrix}
  B_U & 0\\
  B_L & I
 \end{pmatrix}
\]
Observe that 1 is the only eigenvalue of $A_L$ (resp.\ $B_U$) and that neither of
these matrices is diagonalizable (otherwise $(A-I)^2=0$ (resp.\ $(B-I)^2=0$) and so $A$ (resp.\ B) would not have the right Jordan form). Thus we can further conjugate $E_A$ and $E_B$ so that
\[
 A_L=\begin{pmatrix}
      1 & a_{34}\\
      0 & 1
     \end{pmatrix} \hspace{.25 in}
 B_U=\begin{pmatrix}
      1 & 0\\
      b_{21} & 1
     \end{pmatrix}
\]
where $a_{34}\neq 0\neq b_{21}$. Conjugacies that preserve this block form are all of
the form
\[
 \begin{pmatrix}
  u_{11} & 0 & 0 & 0\\
  u_{21} & u_{22} & 0 & 0\\
  0 & 0 & u_{33} & u_{34}\\
  0 & 0 & 0 & u_{44}
 \end{pmatrix}
\]
Finally, a tedious computation\footnote{This computation is greatly expedited by
using Mathematica.} allows us to determine that there exist unique values of the $u_{ij}s$
that will finish putting our matrices in the desired normal form. Note that the
existence of solutions depends on the fact that the entries of $A$ and $B$ are
close to the entries of $f_0(\alpha)$ and $f_0(\beta)$, which live in $\so{3}$. Finally, the entries of all the conjugating matrices are continuous functions of the entries of $A$ and $B$ and so the map $f\mapsto G_f^{-1}fG_f$ is also continuous.    
\end{proof}

Lemma \ref{normformlemma1} gives us the following information about the local dimension of $\charvarpar{F_2}$.
\begin{corollary}\label{charvardim}
 If $\sigma\in \rvarpar{F_2}$ satisfies the hypotheses of Lemma \ref{normformlemma1} then the space $\charvarpar{F_2}$ is locally 5-dimensional at $[\sigma]$.
\end{corollary}

We conclude this section by introducing another normal form for $\so{3}$-parabolics. If we begin with $A$ and $B$ in form \eqref{normform1} then we can
conjugate by the matrix
\[
 V=\begin{pmatrix}
  1 & 0 & 0 & 0\\
  \frac{2-b_{21}-b_{42}}{4} & \frac{2+b_{21}+b_{41}}{4} & 0 & 0 \\
  0 & 0 & \frac{1}{2} & \frac{2b_{32}+b_{21}b_{32}+b_{32}b_{41}-1}{2}\\
  0 & 0 & 0 & \frac{2+b_{21}+b_{41}}{2}
 \end{pmatrix}
\]
and the resulting form will be 
\begin{equation}\label{normform2}
 A=\begin{pmatrix}
  1 & 0 & 1 & a_{14}\\
  0 & 1 & 1 & a_{24}\\
  0 & 0 & 1 & a_{34}\\
  0 & 0 & 0 & 1
 \end{pmatrix},\hspace{.25 in}
B=\begin{pmatrix}
   1 & 0 & 0 & 0\\
   b_{21} & 1 & 0 & 0\\
   b_{31} & 1 & 1 & 0\\
   1 & 1 & 0 & 1
  \end{pmatrix}
\end{equation}
Conjugation by $V$ makes sense because near $f_0$, $2+b_{21}+b_{41}\neq 0$, and so $V$ is non-singular. 

\section{Two Bridge Examples}\label{examples}

In this section we will examine the deformations of various two-bridge knots and links. Two bridge knots and links always admit presentations of a particularly nice form. A two-bridge knot or link, $K$ is determined by a rational number $p/q$, where $q$ is odd, relatively prime to $p$ and $0<q<p$.  There is always a presentation of the form
\begin{equation}\label{twobridgepresentation}
 \pi_1(S^3\bs K)=\langle A,B\mid AW=WB \rangle,
\end{equation}
where $A$ and $B$ are meridians of the knot. The word $W$ can be determined explicitly from the rational number, $p/q$, see \cite{Murasugi61} for details. We now wish to examine deformations of $\Gamma=\pi_1(S^3\bs K)$ that give rise to elements of $\lieslact{\rhogeo}$-cuspidal cohomology. Such a deformation must preserve the conjugacy class of both of the meridians of $\Gamma$, and so we can assume that throughout the deformation our meridian matrices, $A$ and $B$, are of the form \eqref{normform2}. Therefore we can find deformations by solving the matrix equation
\begin{equation}\label{twobridgeequation}
 AW-WB=0
\end{equation}
over $\R$. 

\subsection{Figure-eight knot}\label{fig8section}

The figure-eight knot has rational number $5/3$, and so the word $W=BA^{-1}B^{-1}A$. Solving \eqref{twobridgeequation} we find that 
\begin{equation}\label{fig8solution1}
A=\begin{pmatrix}
1 & 0 & 1 & \frac{3-t}{t-2}\\
0 & 1 & 1 & \frac{1}{(t-2)}\\
0 & 0 & 1 & \frac{t}{2(t-2)}\\
0 & 0 & 0 & 1
\end{pmatrix} \text{ and }
B=\begin{pmatrix}
1 & 0 & 0 & 0\\
t & 1 & 0 & 0\\
2 & 1 & 1 & 0\\
1 & 1 & 0 & 1
\end{pmatrix}
\end{equation}
 and so we have found a 1 parameter family of deformations. However, this family does not preserve the conjugacy class of any non-meridional peripheral element. A longitude of the figure-eight is given by $L=BA^{-1}B^{-1}A^2B^{-1}A^{-1}B$, and a simple computation shows that if \eqref{fig8solution1} is satisfied then $\tr(L)= \frac{48+(t-2)^4}{8(t-2)}$. This curve of representations corresponds to the cohomology classes guaranteed by Theorem \ref{homologyimage}, where $\gamma_1$ can be taken to be the longitude. If we want the entire peripheral subgroup to be parabolic we must add the equation $\tr(L)=4$ to \eqref{twobridgeequation}. When we solve this new set of equations over $\R$, we find that the only solution occurs when $t=4$. 
  Along with Corollary \ref{scdef} this concludes the proof of Theorem \ref{rigiditytheorem} for the figure-eight knot.

\subsection{The Whitehead Link}\label{whiteheadsection}
  
The Whitehead link has rational number $8/3$, and so the word $W=BAB^{-1}A^{-1}B^{-1}AB$. If we again solve \eqref{twobridgeequation} we see that 
\begin{equation}\label{whiteheadsolution}
A=\begin{pmatrix}
1 & 0 & 1 & 0\\
0 & 1 & 1 & -2\\
0 & 0 & 1 & 2\\
0 & 0 & 0 & 1
\end{pmatrix} \text{ and }
B=\begin{pmatrix}
1 & 0 & 0 & 0\\
4 & 1 & 0 & 0\\
-1 & 1 & 0 & 0\\
1 & 1 & 0 & 1
\end{pmatrix}
\end{equation}
 and we find that the Whitehead link has no deformations preserving the conjugacy classes of the boundary elements near the geometric representation. Applying Corollary \ref{scdef} concludes the proof of Theorem \ref{rigiditytheorem} for the Whitehead link. 
 
 This time it was not necessary to place any restriction on the trace of a longitude in order to get a unique solution. Theorem \ref{homologyimage} tells us that there are still 2 dimensions of infinitesimal deformations of $\Gamma$ that we have not yet accounted for. However, we can find deformations that give rise to these extra cohomology classes. To see this, notice that there is a totally geodesic surface that intersects one component of the Whitehead link in a longitude that we can bend along. With reference to Figure \ref{whiteheadfig}, we denote the totally geodesic surface $S$, and the two cusps of the Whitehead link by $C_1$ and $C_2$. If we bend along this surface the effect on $C_1$ will be to deform the meridian (or any non-longitudinal, peripheral curve) and leave the longitude 
fixed. This bending has no effect on $C_2$, since $C_2$ intersects $S$ in two oppositely oriented copies of the meridian, and the 
bending along these two meridians cancel one another (see Theorem \ref{bendingperipheral}). This picture is symmetric and so there is another totally geodesic surface bounding the longitude of $C_2$ and the same argument shows that we can find another family of deformations. 

\begin{figure}\label{whiteheadfig}
  \begin{center}
  \includegraphics[scale=.4]{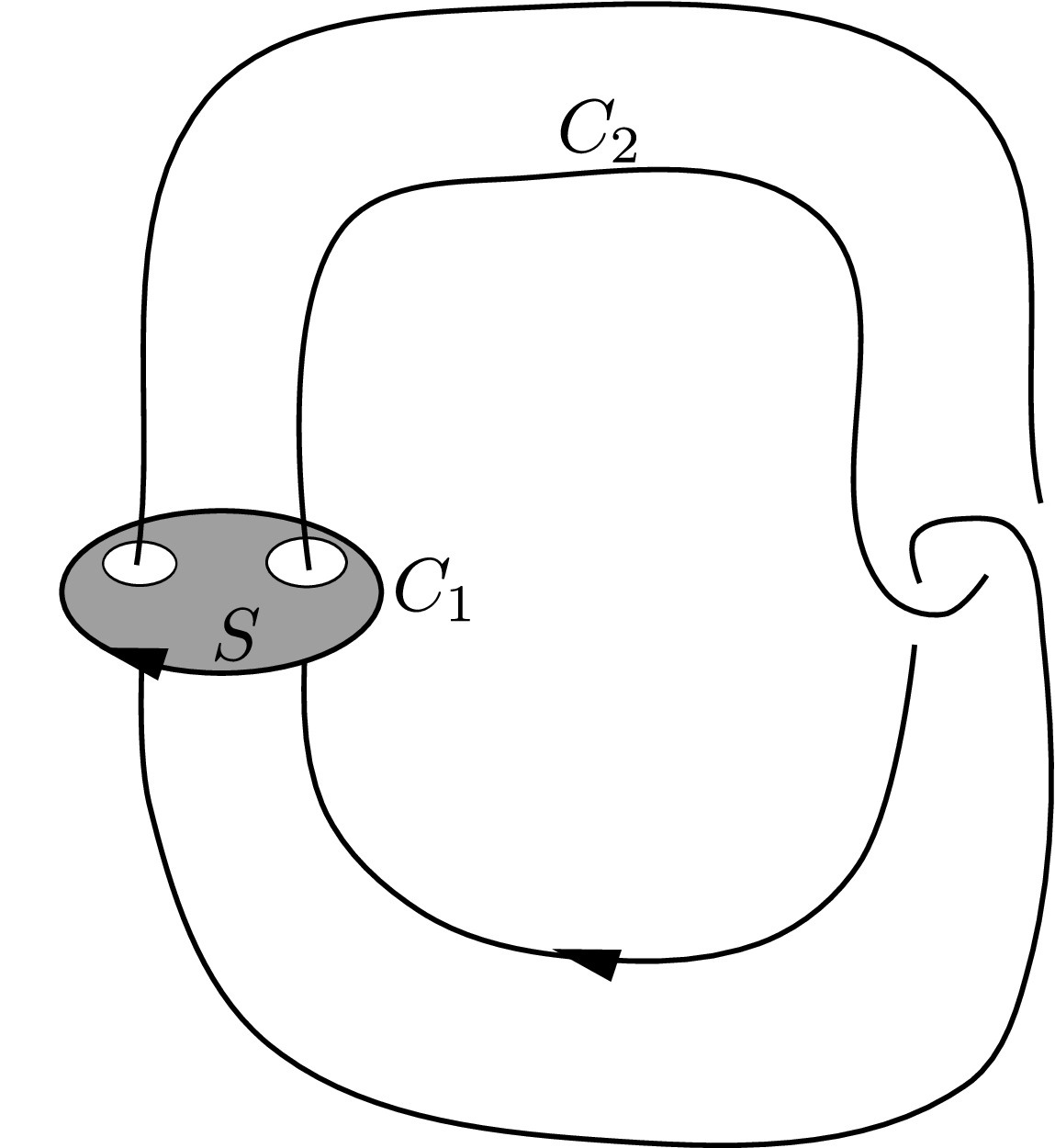}
  \caption{The Whitehead link and its totally geodesic, thrice punctured sphere.}
  \end{center} 
 \end{figure}

\begin{remark}
 Similar computations have been done for the two-bridge knots $5_2$ and $6_1$. However in these two examples it is not known if the cohomology classes coming from Theorem \ref{homologyimage} are integrable. The details of these computations can be found in \cite{Ballas12} and serve to complete the proof of Theorem \ref{rigiditytheorem}.
\end{remark}

\section{Rigidity and Flexibility After Surgery}\label{rigflex} 

In this section we examine the relationship between deformations of a cusped hyperbolic manifold and deformations of manifolds resulting from surgery. The overall idea is as follows: suppose that $M$ is a 1-cusped hyperbolic 3-manifold of finite volume, $\alpha$ is a slope on $\partial M$, and $M(\alpha)$ is the manifold resulting from surgery along $\alpha$. Let $M(\alpha)$ be hyperbolic with geometric representation $\rhogeo$ and $\rho_t$ is a non-trivial family of deformations of $\rhogeo$ into $\pgl$. Since $\pi_1(M(\alpha))$ is a quotient of $\pi_1(M)$ we can pull $\rho_t$ back to a non-trivial family of representations, $\tilde \rho_t$, such that $\tilde\rho_t(\alpha)=1$ for all $t$ (it is important to remember that $\tilde\rho_0$ is not the geometric representation for $M$). In terms of cohomology, we find that the image of the element $\omega\in\homol{M}{\lieslact{\tilde\rho_0}}$ corresponding to $\tilde \rho_t$ has trivial image in $\homol{\alpha}{\lieslact{\tilde\rho_0}}$. With this in mind we will 
call 
slope $\alpha$-\emph{rigid} if the map $\homol{M}{\vact{\rhogeo}}\to \homol{\alpha}{\vact{\rhogeo}}$ is \emph{non-trivial}. Roughly the idea is that if a slope is rigid then we can find deformations that do not infinitesimally fix $\alpha$. The calculations from section \ref{whiteheadsection} show that either meridian of the Whitehead link is a rigid slope.

\subsection{The cohomology of $\partial M$}

Before proceeding we need to understand the structure of the cohomology of the boundary. For details of the facts in this section see \cite[Section 5]{PortiHeusener09}. Let $T$ be a component of $\partial M$ and let $\gamma_1$ and $\gamma_2$ be generators of $\fund{T}$, then $\rhogeo(\gamma_1)$ and $\rhogeo(\gamma_2)$ are both parabolic and can thus be realized as Euclidean translations (along a horosphere). These two translation are determined by vectors $v_{\gamma_1},v_{\gamma_2}\in \R^2$. When the Euclidean angle between $v_{\gamma_1}$ and $v_{\gamma_2}$ is not an integral multiple of $\pi/3$, then we have an injection,
\begin{equation}\label{boundarystructure1}
\homol{T}{\vact{\rho_0}}\stackrel{\iota^\ast_{\gamma_1}\oplus\iota^\ast_{\gamma_2}}{\to}\homol{\gamma_1}{\vact{\rho_0}}\oplus\homol{\gamma_2}{\vact{\rho_0}},
\end{equation}
where $\iota^\ast_{\gamma_i}$ is the map induced on cohomology by the inclusion $\gamma_i\hookrightarrow T$. Additionally, if $\rho_u$ is the holonomy of an incomplete hyperbolic structure, then 
\begin{equation}\label{boundarystructure2}
 \homol[\ast]{\partial M}{\vact{\rho_u}}\cong \homol[\ast]{\partial M}{\R}\otimes\mathfrak{v}^{\rho_u(\fund{\partial M})},
\end{equation}
where $\mathfrak{v}^{\rho_u(\fund{\partial M})}$ are the elements of $\mathfrak{v}$ invariant under $\rho_u(\fund{\partial M}).$ Additionally, for all such representations the image of $\homol{M}{\vact{\rho_u}}$ under $\iota^\ast$ in $\homol{\partial M}{\vact{\rho_u}}$ is 1-dimensional. 

\subsection{Deformations coming from symmetries}

From the previous section we know that the map $\homol{M}{\vact{\rho}}\to\homol{\partial M}{\vact{\rho}}$ has rank 1 whenever $\rho$ is the holonomy of an incomplete hyperbolic structure, and we would like to know how its image sits inside $\homol{\partial M}{\vact{\rho}}$. Additionally, we would like to know when infinitesimal deformations coming from cohomology classes can be integrated to actual deformations. Certain symmetries of $M$ can help us to answer this question. 

Suppose that $M$ is the complement of a hyperbolic amphicheiral knot complement, then $M$ admits an orientation reversing symmetry, $\phi$, that sends the longitude to itself and the meridian to its inverse. The existence of such a symmetry places strong restrictions on the shape of the cusp. Since $\rho$ is the holonomy of a hyperbolic structure, it will factor through a representation into $\slc$ and we wish to examine the image of a peripheral subgroup inside this group. Let $m$ and $l$ be the meridian and longitude of $M$, then it is always possible to conjugate in $\slc$ so that 
$$\rho(m)=
\begin{pmatrix}
 e^{a/2} & 1\\
 0 & e^{-a/2}
\end{pmatrix}\text{ and } 
\rho(l)=
\begin{pmatrix}
 e^{b/2} & \tau_{\rho}\\
 0 & e^{-b/2}
\end{pmatrix}
$$    
The value $\tau_\rho$ is easily seen to be an invariant of the conjugacy class of $\rho$ and we will henceforth refer to it as the \emph{$\tau$ invariant} (see also \cite[App B]{BoileauPorti01}). When $\rho$ is the geometric representation this coincides with the cusp shape. 

Suppose that $[\rho]$ is a representation such that the metric completion of $\HH^3/\rho(\fund{M})$ is the cone manifold $M(\alpha/0)$  (resp.\ $M(0/\alpha))$, where $M(\alpha/0)$  (resp.\ $M(0/\alpha))$ is the cone manifolds obtained by Dehn filling along the meridian (resp.\ longitude) with a solid torus with singular longitude of cone angle $2\pi/\alpha$. Using the $\tau$ invariant we can show that the holonomy of the singular locus of these cone manifolds is a pure translation. 

\begin{lemma}\label{puretranslation}
 Let $M$ be a hyperbolic amphicheiral knot complement and let $\alpha\geq 2$. If $M(\alpha/0)$ is hyperbolic then the holonomy of the longitude is a pure translation. Similarly, if $M(0/\alpha)$ is hyperbolic then the holonomy around the meridian is a pure translation.
\end{lemma} 

\begin{proof}
We prove the result for $M(\alpha/0)$. Provided that $M(\alpha/0)$ is hyperbolic and $\alpha\geq 2$, rigidity results for cone manifolds from \cite{Kojima98,HodgsonKerckhoff98} provides the existence of an element $A_\phi\in \rm{PSL}_2(\C)$ such that $\rho(\phi(\gamma))=\overline{A_\phi\rho(\gamma) A^{-1}_\phi}$, where $\overline{\gamma}$ is complex conjugation of the entries of the matrix $\gamma$. Thus we see that $\tau_{\rho\circ \phi}$ is $\overline{\tau_\rho}$. On the other hand $\phi$ preserves $l$ and sends $m$ to its inverse and so we see that the $\tau_{\rho\circ\phi}$ is also equal to $-\tau_\rho$, and so we see that $\tau_{\rho}$ is purely imaginary.   

The fact that $\rho(m)$ and $\rho(l)$ commute gives the following relationship between $a,b$, and $\tau_\rho$: 
\begin{equation}\label{sinhequation}
 \tau_{\rho} \sinh(a/2)=\sinh(b/2).
\end{equation} 
Using the fact that $\tr(\rho(m))/2=\cosh(a/2)$ and the analogous relationship for $\tr(\rho(l))$ we can rewrite \eqref{sinhequation} as 
\begin{equation}\label{traceequation}
 \tau_{\rho}^2(\tr^2(\rho(m))-4)+4=\tr^2(\rho(l)).
\end{equation}
  For the cone manifold $M(\alpha/0)$, $\rho(m)$ is elliptic and so $\tr^2(\rho(m))\in(0,4)$. As a result we see that $\tr^2(\rho(l))\in(4,\infty)$, and as a result $\rho(l)$ is a pure translation. The proof for $M(0/\alpha)$ is identical with the roles of $m$ and $l$ being exchanged. 
\end{proof}

With this in mind we can prove the following generalization of \cite[Lemma 8.2]{PortiHeusener09}, which lets us know that $\phi$ induces maps on cohomology that act as we would expect. 

\begin{lemma}\label{cohomologyaction}
 Let $M$ be the complement of a hyperbolic, amphicheiral knot with geometric representation $\rhogeo$, and let $\rho_u$ be the holonomy of an incomplete, hyperbolic structure whose completion is $M(\alpha/0)$ or $M(0/\alpha)$, where $\alpha\geq 2$, then 
\begin{enumerate}
 \item $\homol[\ast]{\partial M}{\vact{\rho_u}}\cong \homol[\ast]{\partial M}{\R}\otimes \mathfrak{v}^{\rho_{u}(\fund{\partial M})}$ and $\phi^\ast_u=\phi^\ast\otimes \rm{Id}$, where $\phi_u^\ast$ is the map induced by $\phi$ on $\homol[\ast]{\partial M}{\vact{\rho_u}}$ and $\phi^\ast$ is the map induced by $\phi$ on $\homol[\ast]{\partial M}{\R}$, and 
\item $\iota^\ast_l\circ \phi^\ast_0=\iota^\ast_l$ and $\iota^\ast_m\circ\phi^\ast_0=-\iota^\ast_m$, where $\phi_0^\ast$ is the map induced by $\phi$ on $\homol{\partial M}{\vact{\rhogeo}}$. 
\end{enumerate}

\end{lemma}
\begin{proof}
 The proof that $\homol[\ast]{\partial M}{\vact{\rho_u}}\cong \homol[\ast]{\partial M}{\R}\otimes \mathfrak{v}^{\rho_{u}(\fund{\partial M})}$ is found in \cite[Lem.\ 5.3]{PortiHeusener09}. For the first part we will prove the result for $M(\alpha/0)$ (the other case can be treated identically). By Lemma \ref{puretranslation} we know that $\rho_u(m)$ is elliptic and that $\rho_u(l)$ is a pure translation. Once we have made this observation our proof is identical to the proof given in \cite{PortiHeusener09}. 

For the second part, we begin by observing that by Mostow rigidity there is a matrix $A_0\in \rm{PO}(3,1)$ such that $\rho_0(\phi(\gamma))=A_0\cdot \rho_0(\gamma)$, where the action here is by conjugation. The fact that our knot is amphicheiral tells us that the cusp shape of $M$ is imaginary, and so in $PSL_2(\C)$ we can assume that 
$$\rho_0(m)=
\begin{pmatrix}
 1 & 1\\
0 & 1
\end{pmatrix} {\rm\ and\ }
\rho_0(l)=
\begin{pmatrix}
 1 & i c\\
 0 & 1
\end{pmatrix}
$$
where $c$ is a positive real number. Under the standard embedding of $\rm{PSL}_2(\C)$ into $\pgl$ (as the copy of $\so{3}$ that preserves the form $x_1^2+x^2_2+x_3^2-x_4^2$) we see that 
$$\rho_0(m)={\rm exp}
\begin{pmatrix}
 0 & 0 & 0 & 0\\
 0 & 0 & -1 & 1\\
 0 & 1 & 0 & 0 \\
 0 & 1 & 0 & 0
\end{pmatrix} {\rm\ and\ }
\rho_0(l)={\rm exp}
\begin{pmatrix}
 0 & 0 & -c & c\\
 0 & 0 & 0 & 0\\
 c & 0 & 0 & 0\\
 c & 0 & 0 & 0
\end{pmatrix}
$$
Thus we see that 
$A_0=T
\begin{pmatrix} 
1 & 0 & 0 & 0\\
0 & -1 & 0 & 0 \\
0 & 0 & 1 & 0\\
0 & 0 & 0 & 1
\end{pmatrix},$ where $T$ is some parabolic isometry that fixing the vector $(0,0,1,1)$, which corresponds to $\infty$ in the upper half space model of $\HH^3$. Such a $T$ will be of the form $$T=\rm{exp}
\begin{pmatrix}
 0 & 0 & -a & a\\
 0 & 0 &  -b & b\\
 a & b & 0 & 0 \\
 a & b & 0 & 0
\end{pmatrix}
$$
where $a$ and $b$ are the real and imaginary parts of the complex number that determines the parabolic translation, $T$. Since the cusp shape is imaginary the angle between $m$ and $l$ is $\pi/2$ and we know from \cite[Lemma 5.5]{PortiHeusener09} that the cohomology classes given by the cocycles $z_m$ and $z_l$ generate $H^1(\partial M,\vact{\rho_0})$. Here $z_m$ is given by $z_m(l)=0$ and $z_m(m)=a_l$, where 
$$a_l=\begin{pmatrix}
-1 & 0 & 0 & 0 \\
0 & 3 & 0 & 0 \\
0 & 0 & -1 & 0\\
0 & 0 & 0 & -1
\end{pmatrix}$$
and $z_l$ given by $z_l(m)=0$ and $z_l(l)=a_m$, where 
$$a_m=\begin{pmatrix}
         3 & 0 & 0 & 0 \\
	 0 & -1 & 0 & 0 \\
	 0 & 0 & -1 & 0 \\
	 0 & 0 & 0 & -1
        \end{pmatrix}
$$  
Observe that $\phi_0^\ast z_l(m)=0=z_l(m)$ and that 
$$
 \phi_0^\ast z_m(m)  =A_0^{-1}\cdot z_m(m^{-1}) = -A_0^{-1}\cdot\rho_0(m^{-1})\cdot a_l
.$$
In \cite{PortiHeusener09} it is shown how the cup product associated to the Killing form on $\mathfrak{v}$ gives rise to a non-degenerate pairing. We can now use this cup product to see that $\iota_m^\ast \phi^\ast_0 z_m(m)$ and $-\iota_m^\ast z_m(m)$ are cohomologous. The cup product yields a map $H^1(m,\vact{\rho_0})\otimes H^0(m,\vact{\rho_0})\to H^1(m,\R)\cong \R$ given by $(a\cup b)(m)=B(a(m),b):=8\tr\left(a(m)b\right)$ (where we are thinking of $H^1(m,\R)$ as homomorphisms from $\Z$ to $\R$). Observe that 
$$(\iota_m^\ast \phi^\ast_0z_m\cup a_m)(m)=B(\phi_0^\ast z_m(m),a_m)=32=-B(a_l,a_m)=-(\iota^\ast_m z_m\cup a_m)(m).$$
Since this pairing is non-degenerate and $\left[\iota_m^\ast \phi^\ast_0 z_m\right]$ and $\left[-\iota^\ast_m z_m \right]$ both live in the same 1 dimensional vector space, we see that they must be equal. Thus we see that $\iota^\ast_m\circ\phi^\ast_0=-\iota^\ast_m$. After observing that 
$$(\iota^\ast_l\phi^\ast_0 z_l\cup a_l)(l)=B(\phi^\ast_0 z_l(l),a_l)=-32=B(a_m,a_l)=(\iota_l^\ast z_l\cup l_l)(l),$$
a similar argument shows that $\iota^\ast_l\circ \phi^\ast_0 =\iota^\ast_l$.

\end{proof}

This lemma immediately helps us to answer the question of how the image of $\homol{M}{\vact{\rho_0}}$ sits inside of $\homol{\partial M}{\vact{\rho_0}}$. Since $\phi$ maps $\partial M$ to itself, we see that the image of $\homol{M}{\vact{\rho_0}}$ is invariant under the involution $\phi^\ast$, and so the image will be either the $\pm 1$ eigenspace of $\phi^\ast$. In light of Lemma \ref{cohomologyaction} we see that these eigenspaces sit inside  $\homol{l}{\vact{\rho_0}}$ and $\homol{m}{\vact{\rho_0}}$, respectively. Under the hypotheses of Theorem \ref{achiraldef} the previous fact is enough to show that certain infinitesimal deformations are integrable.

\begin{proof}[Proof of Theorem \ref{achiraldef}]
 In this proof $O=M(n/0)$ and $N$ will denote a regular neighborhood of the singular locus of $M(n/0)$. Hence we can realize $O$ as $M\sqcup N$. By using a Taylor expansion we see that given a smooth family of representations $\sigma_t$ of $\fund[orb]{O}$ into ${\rm SL}_4(\R)$, it is possible to write
\begin{equation}\label{taylorexpansion}\sigma_t(\gamma)=(I+u_1(\gamma)t+u_2(\gamma)t^2+\ldots)\sigma_0(\gamma),
\end{equation}

where the $u_k$ are 1-cochains in $C^1(\fund[orb]{O},\mathfrak{gl}(4)_{\sigma_{0}})$. The fact that each $\sigma_t$ is a homomorphism is satisfied if and only if for each $k$
\begin{equation}\label{integrabilitycondition}
d u_k+\sum_{i=1}^{k-1}u_i\cup u_{k-i}=0, 
\end{equation}
where $a\cup b$ is the 2-cochain given by $(a\cup b)(c,d)=a(c)c\cdot b(d)$, and the action is by conjugation. Conversely, if we are given a cocycle $u_1$ and a collection $\{u_k\}_{k=2}^{\infty}$ such that \eqref{integrabilitycondition} is satisfied for each $k$, then we can apply a deep theorem of Artin \cite[Thm 1.2]{Artin68} to show that we can find an actual deformation $\tilde \sigma_t$ in $\text{GL}_{4}(\R)$ such that the resulting infinitesimal deformation is equal to $u_1$ (here we are using the natural embedding of $\mathfrak{sl}(4)$ into $\mathfrak{gl}(4)\cong\R\oplus \mathfrak{sl}(4)$). We can project this curve to $\pgl$ to find a new curve $\sigma_t$ with infinitesimal deformation $u_1$. In this way we will construct deformations at the level of representations by finding a sequence of 1-cochains satisfying \eqref{integrabilitycondition}.

We first claim that we may assume that for $i=1,2$ that $\homol[i]{O}{\mathfrak{gl}(4)_{\rho_n}}\cong\homol[i]{O}{\vact{\rho_n}}$, where $\rho_n$ is the holonomy of the incomplete structure whose completion is the orbifold $M(n/0)$. Using the splitting of $\mathfrak{gl}(4)\cong\R\oplus\liesl[4]$ we get a splitting of $\homol[i]{O}{\mathfrak{gl}(4)_{\rho_n}}\cong\homol[i]{O}{\R}\oplus\homol[i]{O}{\lieslact{\rho_n}}$. The group $\fund[orb]{O}$ has finite Abelianization and so we see that $\homol{O}{\R}=0$. By duality we see that $\homol[2]{O}{\R}=0$ as well.  Using Weil rigidity and the splitting \eqref{splitting} we see that $\homol{M(n/0)}{\lieslact{\rho_n}}\cong \homol{M(n/0)}{\vact{\rho_n}}$. Finally, duality tells us that $\homol[2]{O}{\lieslact{\rho_n}}\cong \homol[2]{O}{\vact{\rho_n}}$.

In order to simplify notation $\homol[\ast]{\underline{\ \ }}{\mathfrak{gl}(4)_{\rho_n}}$ will be denoted $H^\ast(\underline{\ \ })$. The orbifold $O$ is finitely covered by an aspherical manifold, and so by combining a 
transfer argument \cite{PortiHeusener09} with the fact that CW-cohomology with twisted coefficients and group cohomology with twisted coefficients coincide for aspherical manifolds \cite{Whitehead78}, we can conclude that group cohomology with twisted coefficients for $\fund[orb]{O}$ is the same as twisted CW-cohomology with twisted coefficients for the orbifold, $O$. Therefore, we can use a Mayer-Vietoris sequence to analyze cohomology. Consider the following section of the sequence.

\begin{equation}\label{exactsequence1}
H^0(M)\oplus H^0(N)\to H^0(\partial M)\to  H^1(O)\to H^1(M)\oplus H^1(N)\to H^1(\partial M).
\end{equation}

Next, we will determine the cohomology of $N$. Since $N$ has the homotopy type of $S^1$ it will only have cohomology in dimension 0 and 1. Since $\rho_n(\partial M)=\rho_n(N)$ we see that $H^0(N)\cong H^0(\partial M)$ (both are 1-dimensional). Finally by duality we see that $H^1(N)$ is also 1-dimensional. Additionally, $H^0(O)$ is trivial since $\rho_n$ is irreducible. Combining these facts, we see that the first arrow of \eqref{exactsequence1} is an isomorphism and thus  the penultimate arrow of \eqref{exactsequence1} is injective. We also learn that $H^1(O)$ injects into $H^1(M)$, since if a cohomology class from $H^1(O)$ dies in $H^1(M)$ then exactness tells us that it must also die when mapped into $H^1(N)$ (since $H^1(N)$ injects into $H^1(\partial M)$). However, this contradicts the fact that $H^1(O)$ injects into $H^1(M)\oplus H^1(N)$. Using the fact that the longitude is a rigid slope, Lemma \ref{cohomologyaction}, and \cite[Cor 6.6]{PortiHeusener09} we see that $\phi^\ast$ acts as the identity on 
$H^1(O)$. Since $H^1(M)$ 
and $H^1(N)$ have the 1-eigenspace of $\phi^\ast$ as their image in $H^1(\partial M)$, the last arrow of \eqref{exactsequence1} is not a surjection, and so $H^1(O)$ is 1-dimensional.

Duality tells us that $H^2(O)$ is 1-dimensional and we will now show that $\phi^\ast$ act as multiplication by -1. Again, by duality we see that $H^3(O)=0$ and so the Mayer-Vietoris sequence contains the following piece.

\begin{equation}\label{exactsequence2}
H^1(M)\oplus H^1(N) \to H^1(\partial M)\to H^2(O)\to H^2(M)\oplus H^2(N)\to H^2(\partial M)\to 0.
\end{equation}

Since $H^2(O)$ is 1-dimensional, the second arrow of \eqref{exactsequence2} is either trivial or surjective. If this arrow is trivial then the third arrow is an injection and thus an isomorphism for dimensional reasons. However, this is a contradiction since the penultimate arrow is a surjection and $H^2(\partial M)$ is non-trivial. Since the first arrow of \eqref{exactsequence2} has the 1-eigenspace of $\phi^\ast$ as its image we see that the -1-eigenspace of $\phi^\ast$ surjects $H^2(O)$. However, since the Mayer-Vietoris sequence is natural and $\phi$ respects the splitting of $O$ into $M \cup N$ we see that $\phi^\ast$ acts on $H^2(O)$ as -1. 

We can now construct the sequence  $\{u_k\}_{k=2}^\infty$ satisfying \eqref{integrabilitycondition}. Let $[u_1]$ be a generator of $H^1(O)$ (we know this is one dimensional by the previous paragraph). Since $\phi$ is an isometry of a finite volume hyperbolic manifold we know that it has finite order when viewed as an element of $\rm{Out}(\fund{M})$ \cite{ThurstonNotes}, and so there exists a finite order map $\psi$ that is conjugate to $\phi$. Because the maps $\phi$ and $\psi$ are conjugate, they act in the same way on cohomology \cite{Brown82}. Let $L$ be the order of $\psi$. Since $\psi$ acts as the identity on $H^1(O)$, the cocycle
$$u_1^\ast=\frac{1}{L}(u_1+\phi(u_1)+\ldots +\phi^{L-1}(u_i))$$
is both invariant under $\psi$ and cohomologous to $u_1$. By replacing $u_1$ with $u_1^\ast$ we can assume that $u_1$ is invariant under $\psi$. Next observe that 
$$-u_1\cup u_1\sim \psi(u_1\cup u_1)=\psi(u_1)\cup\psi(u_1)=u_1\cup u_1,$$
and so we see that $u_1\cup u_1$ is cohomologous to 0, and so there exists a 1-cochain, $u_2$, such that $du_2+u_1\cup u_1 =0$. Using the same averaging trick as before we can replace $u_2$ by the $\psi$-invariant cochain $u_2^\ast$. By invariance of $u_1$ we see that $u_2^\ast$ has the same boundary as $u_2$ so this replacement does not affect the first part of our construction. Again we see that 
$$-(u_1\cup u_2+u_2\cup u_1)\sim \psi(u_1\cup u_2 +u_2\cup u_1)=u_1\cup u_2+u_2\cup u_1,$$
and so there exist $u_3$ such that \eqref{integrabilitycondition} is satisfied. Repeating this process indefinitely, we can find the sequence $\{u_k\}$ satisfying \eqref{integrabilitycondition}, and thus by Artin's theorem \cite[Thm 1.2]{Artin68} we can find our desired deformation. 

Let $\rho_t$ be the family of representations that we have just constructed. By work of Koszul \cite{Koszul68} and Benoist \cite{Benoist05} we see each of these representations will be discrete and faithful. Additionally, for each $t$ we can find a properly convex domain, $\Omega_t$, that is preserved by $\rho_t$ and such that $O\cong \Omega_t/\rho_t(\Gamma)$. Furthermore, the group $\fund[orb]{O}$ is word hyperbolic and thus by work of Benoist \cite{Benoist04I} we see that the $\Omega_t$ are actually strictly convex. We have thus found the desired family of strictly convex deformations of the complete hyperbolic structure on $O$.  
     
\end{proof}

\bibliographystyle{plain}
\bibliography{bibliography}

\end{document}